\documentclass[a4paper,reqno,12pt]{amsart}

\usepackage{hyperref}
\usepackage{amsmath}
\usepackage{amsthm}
\usepackage[T1]{fontenc}
\usepackage{amsmath}
\usepackage{graphicx}
\usepackage{epsf, subfigure, verbatim}
\usepackage{epsfig}
\usepackage{amssymb}
\usepackage{amsfonts}
\usepackage{empheq}
\usepackage{float}

\usepackage[margin=1.50in]{geometry}
\usepackage{srcltx}

\theoremstyle{plain} 

\newtheorem{theorem}[equation]{Theorem}

\newtheorem{lemma}[equation]{Lemma}

\numberwithin{equation}{subsection}
\newtheorem{question}[equation]{Question}
\newtheorem{remark}[equation]{Remark}

\newcommand{\lmoh}{\lambda^{-\frac{1}{2}}}

\newcommand{\loh}{\lambda^{\frac{1}{2}}}

\begin{document}

\title[Average $L^q$ growth and nodal sets of eigenfunctions on surfaces]{Average $L^q$ growth and nodal sets of eigenfuntions of the Laplacian on surfaces}
\author{Guillaume Roy-Fortin}
\thanks{The author has been supported by NSERC}

\begin{abstract}
In \cite{RF}, we exhibit a link between the average local growth of Laplace eigenfunctions on surfaces and the size of their nodal set. In that paper, the average local growth is computed using the uniform - or $L^\infty$ - growth exponents on disks of wavelength radius. The purpose of this note is to prove similar results for a broader class of $L^q$ growth exponents with $q \in (1, \infty)$. More precisely, we show that the size of the nodal set is bounded above and below by the product of the average local $L^q$ growth with the frequency. We briefly discuss the relation between this new result and Yau's conjecture on the size of nodal sets. 
\end{abstract}

\maketitle

\section{Growth and nodal sets of Laplace eigenfunctions}
Let $(M,g)$ be a smooth, compact, connected two-dimensional Riemannian manifold without boundary endowed with a $C^\infty$ metric $g$. Let $\{\phi_{\lambda}\}$, $\lambda \nearrow \infty$, be any sequence of $L^2$ normalized eigenfunctions of the negative definite Laplace-Beltrami operator $\Delta_g$: 
\begin{equation}\label{main_eigenvalue_eq}
\Delta_g \phi_\lambda + \lambda \phi_\lambda = 0.
\end{equation}

\subsection{Nodal set}
The nodal set of an eigenfunction $\phi_\lambda$ is $$Z_\lambda = \left\{ p \in M : \phi_\lambda (p) = 0 \right\},$$ whose one dimensional Hausdorff measure we denote by $\mathcal{H}^1(Z_\lambda)$. The nodal set is a smooth curve away from the finite singular set $$S_\lambda = \left\{ p \in Z_\lambda : \nabla \phi_\lambda(p) = 0 \right\},$$ which is known to be finite in our current setting (see \cite{Ch, B}). It has been conjectured by Yau \cite{Ya1, Ya2} that the size of the nodal set grows like the frequency, namely that there exist positive constants $c,C$ such that: 
\begin{equation}\label{Yau}
c \sqrt{\lambda} \leq \mathcal{H}^1(Z_\lambda) \leq C  \sqrt{\lambda}.
\end{equation}

\begin{remark} The conjecture of Yau has been formulated for any n-dimensional compact, smooth manifold and has been proved by Donnelly and Fefferman in \cite{DF1} in the case where $(M,g)$ is a real analytic manifold with real analytic metric. We will however only consider the case $n=2$ from here on.
\end{remark} 

For smooth surfaces, the conjectured lower bound has been proved in \cite{Br} by Br{\"u}ning and also by Yau (unpublished). The current best upper bound of 
\begin{equation}\label{bestUB_YauSurface}
\mathcal{H}^1(Z_\lambda) \leq C  \lambda^{\frac{3}{4}}
\end{equation}
is due to Donnelly-Fefferman \cite{DF2} and Dong \cite{D}. 

\subsection{Local growth}

Given a continuous function $f$ on a ball $B$ and a scaling factor $0 < \alpha < 1$, one can measure the local growth of $f$ by defining the \textit{$L^q$ growth exponent} of $f$ on $B$ by $$\beta^q_\alpha(f; B) := \log \frac{|| f ||_{L^q(B)}}{|| f ||_{L^q(\alpha B)}},$$ where $\alpha B$ is the ball concentric to $B$ with radius shrank by a factor $\alpha$. This quantity can be thought of as a generalized degree for $f$. Indeed, the following basic example illustrates that the $L^q$ growth exponent of a polynomial is nothing but its degree up to constants: 
$$\beta^q_\alpha(x^n; [-1,1]) = q^{-1} \log \frac{ \int_{-1}^1 |x|^{qn} dx }{\int_{-\alpha}^\alpha |x|^{qn} dx } = q^{-1} \log \frac{1}{\alpha^{qn+1}} = c(\alpha) (n + q^{-1}).$$
We now define growth exponents at small scale for the eigenfunctions $\phi_\lambda$. Once again, fix a scaling factor $0 < \alpha < 1$ and write $B_{r_\lambda}(p)$ for the ball of radius 
\begin{equation}\label{eq_k0}
r_\lambda = k_0 \lmoh
\end{equation}
centred at $p \in M$. Here, $k_0$ is a small constant whose value will be determined later. For $q \in [1, \infty]$, we define 
$$\beta^q_\alpha(\lambda,p) := \beta^q_\alpha(\phi_\lambda; B_{r_\lambda}(p)).$$
These growth exponents are a measure of the local $L^q$ growth of an eigenfunction at the wavelength scale and generalize the doubling exponents extensively used by Donnelly and Fefferman, notably in \cite{DF1}, where they prove the following estimate $$ \beta^q_\alpha(\lambda,p) \leq c \sqrt{\lambda}.$$
Note that this bound further supports the common intuition that an eigenfunction of eigenvalue $\lambda$ behaves roughly like a polynomial of degree $\sqrt{\lambda}$. For more details, we refer the reader to \cite{DF3, RF, Ze2}. The growth exponents are local by nature and we can average them to get a global quantity called the \textit{average local $L^q$ growth}: $$A^q_{\alpha}(\lambda) := \frac{1}{\text{Vol}(M)} \int_M \beta^q_\alpha(\lambda,p) \mathrm{d} V_g(p).$$ 
\section{Main result and discussion}
Our main result shows that the length of the nodal set of an eigenfunction $\phi_\lambda$ is controlled by the product of the frequency with the average local $L^q$ growth of $\phi_\lambda$. More precisely, we have

\begin{theorem}\label{CHAP2_thm1}
There exists $0 < \alpha_0 < 1$ such that the following holds for any $0 < \alpha < \alpha_0$ and $q \in (1, +\infty)$: 
$$c_1 \loh A^q_\alpha(\lambda) \leq \mathcal{H}^1(Z_\lambda) \leq c_2 \loh (A^q_\alpha(\lambda) + 1),$$
where  and $c_1, c_2$ are positive constants depending only on $q$, $\alpha$ and the geometry of $(M,g)$.
\end{theorem}
In our previous paper \cite{RF}, we prove the same result for $q=\infty$. We remark that the definition of the growth exponents implies that the lower bound for the length of the nodal set actually holds for every $\alpha \in (0,1)$. Indeed, as the scaling parameter $\alpha$ increases, the growth exponents become smaller. We believe that the case $q=1$ is also true, but we can not prove it with our current methods.

\subsection{Connection with the conjecture of Yau}
We recall that for smooth surfaces, the sharp lower bound conjectured by Yau for the length of the nodal set has been proved and that the current best upper bound is $\lambda^{\frac{3}{4}}$. We ask

\begin{question}\label{question1}
Is it possible to prove a polynomial (or better) upper bound for the average local $L^q$ growth of the type $$A^q_\alpha(\lambda) = O(\lambda^\delta),$$ with $\delta \in [0, \frac{1}{4}[$ and for some $q \in (1, \infty)$? 
\end{question}

Combined with Theorem \ref{CHAP2_thm1}, any such result would immediately improve the current best upper bound for the size of the nodal set on smooth surfaces. In the same spirit, we remark that another consequence of our main theorem is that the conjecture of Yau for compact surfaces is now equivalent to proving $$A^q_\alpha(\lambda) = O(1),$$ for any $q \in (1, \infty]$. Finally, for eigenfunctions of a real analytic surface, Theorem \ref{CHAP2_thm1} combined with the results of Donnelly and Fefferman \cite{DF1} imply that $A^q_\alpha(\lambda) = O(1).$ This tells us that such eigenfunctions cannot grow too fast with respect to $L^q$ norms on balls of wavelength radius, except maybe on a residual set of null measure. \\

To tackle Question \ref{question1}, we need to further understand what exactly is measured by $A^q_\alpha(\lambda)$, which we attempt to do next in the special case $q=2$.

\subsection{Remark on equidistribution of eigenfunctions and the case $q=2$}
A subsequence $\{ \phi_{j_k}\}_{k=1}^\infty$ of $L^2$-normalized eigenfunctions is \textit{equidistributed} on some set $E \subset M$ if its $L^2$-mass on $E$ converges to Vol($E$), that is, if $$ \lim_{k \rightarrow \infty} \int_E |\phi_{j_k}|^2 \mathrm{d}V = \frac{\text{Vol}(E)}{\text{Vol}(M)}.$$
Equidistribution on $M$ often arises as a consequence of the stronger quantum ergodicity property, where the eigenfunctions actually equidistribute on the phase space $S^*M$ after microlocal lifting.  For example, on surfaces with negative curvature, equidistribution holds for a density one subsequence of quantum ergodic eigenfunctions, see \cite{CdV, Sh, Ze3}. The recent papers \cite{HR} by Hezari, Rivi\`ere and \cite{H} by Han investigate quantum ergodicity of eigenfunctions at small scales on closed manifolds of negative sectional curvature. A consequence of their work is that, in such a setting, the full density subsequence of quantum ergodic eigenfunctions equidistribute on balls $B_r(p)$ of shrinking radius $r = (\log{\sqrt{\lambda}})^{-K}$. This relates to our work when $q=2$: the definition of $L^2$-growth exponents relies upon balls whose radii are also shrinking, albeit at the quicker wavelength pace $r=\lmoh$. While the current machinery does not seem to allow going beyond the inverse logarithmic regime in general, it would nevertheless be interesting to try to find specific sequences of eigenfunctions that equidistribute almost everywhere at the wavelength scale. For such sequences, the average local growth is bounded and the upper bound conjectured by Yau for the length of the nodal set would then follow from Theorem \ref{CHAP2_thm1}.\\

The aforementioned results on small scale quantum ergodicity imply that the $L^2$ growth exponents are uniformly bounded at the inverse logarithmic scale for the density one subsequence of quantum ergodic eigenfunctions on surfaces of negative curvature. It is thus natural to ask
\begin{question}
Does a result analogous to Theorem 1 hold for growth exponents measured on balls of larger scales than $r=\lmoh$ on surfaces of negative curvature? 
\end{question}

Proving such a result for $r = (\log{\sqrt{\lambda}})^{-K}$ could lead to an improvement of the upper bound for the length of the nodal set of quantum ergodic eigenfunctions. 

\subsection{Acknowledgements.}
This research was part of my Ph.D. thesis at Universit\'e de Montr\'eal under the supervision of Iosif Polterovich and I want to thank him for his always precious input and constant support. I want to thank Leonid Polterovich, Misha Sodin, John Toth and Steve Zelditch for valuable discussions related to the topics involved in this paper. 

\section{Proofs}

\subsection{Upper bound for the size of the nodal set}

The proof of the upper bound for the size of the nodal set follows that of Section 2.2 in \cite{RF}, where we replace $L^\infty$ norms of the eigenfunctions $\phi_\lambda$ by $L^q$ norms throughout. In that section, the eigenfunction is locally represented by a function $F$ which solves a planar Schr\"{o}dinger equation with small potential. This is done by restricting the eigenfunction to a small ball within a conformal coordinate patch. We then use \cite[Theorem 2.1.1]{RF} which suitably relates the $L^\infty$ growth exponent of $F$ with the size of its nodal set and conclude using a integral geometric argument. Thus, we only need the following $L^q$ analogue of that theorem:

\begin{theorem}\label{theorem2}

Let $F: 3\mathbb{D} \rightarrow \mathbb{R}$ be a solution of 
\begin{equation}\label{eq_schroedinger_ev_3D}
\Delta F + pF = 0, 
\end{equation}
with the potential $p \in C^{\infty}(3\mathbb{D})$ satisfying $||p||_{L^\infty(3 \mathbb{D})}  < \epsilon_0$. Let $q \in (1, \infty)$ and let $$\beta^q := \log \frac{||F||_{L^q(\frac{11}{4} \mathbb{D})}  }{||F||_{L^q(\frac{1}{4} \mathbb{D})} }.$$ Finally, denote by $Z_F$ the nodal set $\{ z \in 3 \mathbb{D} : F(z) = 0 \}$ of $F$. Then,
\begin{equation*}
\mathcal{H}^1 \left( Z_F \cap \frac{1}{60}\mathbb{D} \right) \leq c_3 \beta^*,
\end{equation*}
where $\beta^* := \max\{ \beta^q, 1 \}$ and $c_3 = c_3(q)$ is a positive constant.
\end{theorem}

The explicit value of the small positive constant $\epsilon_0$ comes from the proof. Choosing $k_0$ small enough in Equation \ref{eq_k0} ensures that the restriction of eigenfunctions to balls of radius $r=k_0 \lmoh$ within a conformal patch will give rise to a family of functions $F$ that all satisfy the hypotheses of Theorem \ref{theorem2}. Also, remark that the radius $\frac{11}{4}$ of the bigger disk for the new $L^q$ growth exponent is slightly larger than the corresponding disk in the $L^\infty$ growth exponent, whose radius is $\frac{5}{2}.$ These values are arbitrary and this does not affect the global argument. We first start with another lemma, which allows to bound the $L^\infty$ norm of $F$ on a disk by its $L^q$ norm on a slightly larger one. 

\begin{lemma}\label{CHAP2_lemma2}
Let $F: r_0 \mathbb{D} \rightarrow \mathbb{R}$ be a solution of $\Delta F + p F = 0$ on $r_0 \mathbb{D}$, where $r_0 > 0$ is a fixed radius and $p \in C^\infty(r_0\mathbb{D})$ is a small potential which satisfies $ ||p||_{L^\infty(r_0 \mathbb{D})} < \epsilon_0$. Let $q \in (1, \infty)$ and consider the following radii $0 < r^- < r^+ \leq r_0$. Then, $$||F||_{L^\infty \left(r^- \mathbb{D} \right)} \leq c_4  ||F||_{L^q\left(r^+ \mathbb{D}\right)}, $$ where $c_4$ is a positive constant that depends on the choice of $r^-, r^+$ and the exponent $q$. 
\end{lemma}
\begin{proof}
The proof uses ideas from \cite[Lemma 4.9]{NPS} and generalizes \cite[Lemma 5.4.6]{RF}. The main tool is the representation of $F$ as the sum of its Green potential and Poisson integral. More precisely, for $|z| \leq r^-$ and given any fixed radius $\rho \in [\tilde{r}, r^+]$, with $\displaystyle \tilde{r} := \frac{r^- + r^+}{2}$, we have the following decomposition of F:
\begin{equation}\label{CHAP2_lemma_prop1_doub_exp_eq3}
F(z) = \iint\limits_{\rho\mathbb{D}} p(\zeta) F(\zeta) G_\rho(z,\zeta) dA(\zeta) + \int\limits_{\rho \,\mathbb{S}^1} F(\zeta)  P_\rho(z, \zeta) ds(\zeta),
\end{equation}
where $\displaystyle G_\rho(z, \zeta) = \log{\left| \frac{\rho^2 - z\bar{\zeta}}{\rho(z - \zeta)} \right|}$ and $\displaystyle P_\rho(z, \zeta) = \frac{\rho^2 -|z|^2}{|\zeta - z|^2}$.
We respectively write $I_1$ and $I_2$ for the double integral and the (line) integral above. Since $q > 1$, the convexity of $ x \mapsto x^q$ yields
\begin{equation}\label{CHAP2_lemma_prop1_doub_exp_eq4}
|F(z)|^q = | I_1 + I_2 |^q \leq 2^{q-1}( |I_1|^q + |I_2|^q),
\end{equation}
which holds for all $|z| \leq r^-$. Let $q' = \frac{q}{q-1}  < \infty$ be the conjugate exponent of $q$. By H\"{o}lder, we have
\begin{align}
|I_1|^q &\leq  \iint\limits_{\rho \mathbb{D}} |p(\zeta)|^q |F(\zeta)|^q dA(\zeta)  \left( \;\; \iint\limits_{\rho \mathbb{D}} |G_\rho(z,\zeta)|^{q'} dA(\zeta)\right)^{q-1}\notag\\
&\leq a_1(q) \iint\limits_{\rho \mathbb{D}}  |p(\zeta)|^q |F(\zeta)|^q dA(\zeta) \leq a_1(q) ||p||_{\infty}^q  \iint\limits_{\rho \mathbb{D}} |F(\zeta)|^q dA(\zeta) \notag\\
&\leq a_1(q)\, \epsilon_0  \iint\limits_{r^+ \mathbb{D}} |F(\zeta)|^q dA(\zeta), \label{CHAP2_lemma_prop1_doub_exp_eq5}
\end{align} In the above, we have bounded the $L^{q'}$ norm of the Green function by $$ a_1(q) =  \sup_{|z| \leq r^-}  \left ( \;\, \iint\limits_{r^+ \mathbb{D}} \left( \log \frac{(r^+)^2 + |z\bar{\zeta}|}{\tilde{r}|z - \zeta|} \right)^{q'} dA(\zeta) \, \right)^{q-1}.$$ Note as well that we assumed $\epsilon_0 < 1$ without loss of generality. The actual size of $\epsilon_0$ will be specified at the end of the proof. We proceed similarly for the Poisson integral and get
\begin{equation}\label{CHAP2_lemma_prop1_doub_exp_eq6}
|I_2|^q \leq \int\limits_{\rho \,\mathbb{S}^1} |F(\zeta)|^q ds(\zeta) \left( \int\limits_{\rho \,\mathbb{S}^1} |P_\rho(z,\zeta)|^{q'} ds(\zeta)\right)^{q-1} \leq a_2(q) \int\limits_{\rho \,\mathbb{S}^1} |F(\zeta)|^q ds(\zeta),
\end{equation}
with $$ a_2(q) =  \left( \; \int\limits_{r^+ \,\mathbb{S}^1}  \left( \frac{r^+}{\tilde{r} - r^-}\right)^{2q'} ds \right)^{q-1} = (2\pi r^+)^{q-1}   \left( \frac{r^+}{\tilde{r} - r^-}\right)^{2q}.$$
 The representation of $F$ in Equation $\ref{CHAP2_lemma_prop1_doub_exp_eq3}$ holds for any $|z| \leq r^-$ so that substituting Equations $\ref{CHAP2_lemma_prop1_doub_exp_eq5}$ and $\ref{CHAP2_lemma_prop1_doub_exp_eq6}$ in $\ref{CHAP2_lemma_prop1_doub_exp_eq4}$, we get:
\begin{equation*}
\sup_{|z| \leq r^-} |F|^q \leq a_3(q) \left( \epsilon_0 \iint\limits_{r^+\mathbb{D}} |F|^q dA + \int\limits_{\rho \,\mathbb{S}^1} |F|^q ds\right), \; \forall \rho \in [\tilde{r},r^+],
\end{equation*}
with $a_{3} = 2^{q-1} \max\{a_1(q), a_2(q)\}.$ Averaging over all $\rho$ yields:
\begin{align}
\sup_{|z| \leq r^-} |F|^q &\leq a_4(q) \left( \epsilon_0 \iint\limits_{r^+\mathbb{D}} |F|^q dA +\iint\limits_{\tilde{r} < |z| < r^+} |F|^q dA\right) \notag\\
&\leq a_4(q)  \left( \epsilon_0 \iint\limits_{r^- \mathbb{D}} |F|^q dA + (1 + \epsilon_0) \iint\limits_{r^- < |z| < r^+} |F|^q dA \right) \notag\\
&\leq a_5(q) \epsilon_0 \sup_{|z| \leq r^-} |F|^q + a_4(q)(1 + \epsilon_0) \iint\limits_{r^+ \mathbb{D}} |F|^q dA. \label{CHAP2_lemma_prop1_doub_exp_eq7}
\end{align}
Hence,
\begin{equation*}
(1 - a_5(q) \epsilon_0) \sup_{|z| \leq r^-} |F|^q \leq a_4(q)(1 + \epsilon_0) \iint\limits_{r^+\mathbb{D}} |F|^q dA.
\end{equation*}
It suffices to choose $\epsilon_0$ small enough so that $(1 - a_5(q) \epsilon_0)$ is positive to finally obtain
\begin{equation*}\label{CHAP2_lemma_prop1_doub_exp_eq8}
\sup_{|z| \leq r^-} |F|^q \leq \frac{a_4(q)(1 + \epsilon_0)}{1 - a_5(q) \epsilon_0} \iint\limits_{r^+ \mathbb{D}} |F|^q dA,
\end{equation*}
whence we conclude $$||F||_{L^\infty(r^- \mathbb{D})} \leq c_4 ||F||_{L^q(r^+ \mathbb{D})}.$$
\end{proof}
Let us remark here that it is also possible to prove the last result for $q \in [2, \infty)$ using classical elliptic theory, as is extensively used by Donnelly and Fefferman in \cite{DF1, DF2}, an approach which works in higher dimension to the cost of being more complicated than what we have just done here. We are now ready to give the proof of Theorem \ref{theorem2}.
\begin{proof}
We set $\displaystyle r_0 = 3, r^- = \frac{5}{2}$ and $r^+= \frac{11}{4}$ and use Lemma \ref{CHAP2_lemma2} to get $$|| F ||_{L^\infty \left(  \frac{5}{2} \mathbb{D}\right) } \leq c_4 ||F||_{L^q \left( \frac{11}{4} \mathbb{D}\right) }.$$
Also, we have that $ \displaystyle || F ||_{L^q \left(  \frac{1}{4} \mathbb{D}\right) }  \leq a_1 || F ||_{L^\infty \left(  \frac{1}{4} \mathbb{D}\right) }.$ Hence,
$$ \frac{ || F ||_{L^\infty \left(  \frac{5}{2} \mathbb{D}\right) } } { || F ||_{L^\infty \left(  \frac{1}{4} \mathbb{D}\right) } } \leq  a_2 \frac{ || F ||_{L^q \left(  \frac{11}{4} \mathbb{D}\right) } } { || F ||_{L^q \left(  \frac{1}{4} \mathbb{D}\right) } }.$$  We conclude by taking the logarithm on both sides and by using Theorem 2.1.1 of \cite{RF}.
\end{proof}
\subsection{Lower bound for the size of the nodal set}
The approach is similar to what we just did for the upper bound: we now follow the steps of Section 3.3 in \cite{RF} using $L^q$ norms instead of $L^\infty$ ones and we replace the important Theorem 3.1.1. by the following theorem

\begin{theorem}
Let $F: \overline{\mathbb{D}} \rightarrow \mathbb{R}$ be a solution of 
\begin{equation}\label{eq_schroedinger_ev_1D}
\Delta F + pF = 0,
\end{equation}
in $\mathbb{D}$ and with the potential $p \in C^{\infty}(\mathbb{D})$ satisfying $||p||_{L^\infty(\overline{\mathbb{D}})} < \epsilon_0$. Denote by $|Z_F(\mathbb{S}^1)|$ the number of zeros of $F$ on the unit circle $\mathbb{S}^1$ and let $q \in (1, \infty)$. Then, 
\begin{equation*}
\log \frac{||F||_{L^q(\rho^+ \mathbb{D})}}{||F||_{L^q(\tilde{\rho}^- \mathbb{D})}} \leq c_5(1 + |Z_F(\mathbb{S}^1)|),
\end{equation*}
where $0 < \tilde{\rho}^- < \rho^+ < \frac{1}{2}$ are fixed, small radii and $c_5 = c_5(q)$ a positive constant.
\end{theorem}

Notice that the explicit value of $\tilde{\rho}^-$ above is slightly bigger than that of $\rho^-$ in Theorem 3.1.1 of \cite{RF}, but, again, this has no effect whatsoever on the global argument.

\begin{proof}
On the one hand, we have 
\begin{equation}\label{CHAP2_lemma2_eq1}
||F||_{ L^q( \rho^+ \mathbb{D} ) } = \left( \, \int\limits_{\rho^+ \mathbb{D}} |F|^q \,\mathrm{d} A\right)^\frac{1}{q} \leq \left( \pi (\rho^+)^2 \right)^\frac{1}{q} \left( \sup_{\rho^+ \mathbb{D}} |F|^q \right)^{\frac{1}{q}} \leq ||F||_{L^\infty(\rho^+ \mathbb{D})}.
\end{equation}

On the other hand, applying Lemma \ref{CHAP2_lemma2} directly yields $$||F||_{L^\infty(\rho^- \mathbb{D})}  \leq c_4 ||F||_{L^q(\tilde{\rho}^- \mathbb{D})}.$$

It suffices to combine the last equations and invoke \cite[Theorem 3.1.1]{RF} in order to conclude the proof.
 
\end{proof}
%
%
%

\bigskip

{\scshape Department of Mathematics,
Northwestern University, 2033 Sheridan Rd., Evanston, IL
60208}

\emph{E-mail address:} \verb"gui@math.northwestern.edu"

\end{document}